\RequirePackage{fix-cm}
\documentclass[smallcondensed]{svjour3}     
\smartqed  
\usepackage{mathptmx}      
\usepackage{amsmath}
\usepackage{latexsym}
\usepackage{graphicx}
\usepackage{xhfill}
\newtheorem{algorithm}{Algorithm}

\begin{document}
\newcommand{\xfill}[2][1ex]{{%
  \dimen0=#2\advance\dimen0 by #1
  \leaders\hrule height \dimen0 depth -#1\hfill%
}}
\title{Parallel extragradient - viscosity methods for equilibrium problems and fixed point problems
}

\titlerunning{Parallel extragradient - viscosity methods for EPs and FPPs}        

\author{Dang Van Hieu   
}

\authorrunning{D.V. Hieu} 
\institute{Dang Van Hieu \at
              Department of Mathematics, Vietnam National University,  \\
              334 Nguyen Trai, Thanh Xuan, Hanoi, Vietnam \\
              \email{dv.hieu83@gmail.com}           
}

\date{Received: date / Accepted: date}

\maketitle

\begin{abstract}
In this paper, we propose two parallel extragradient - viscosity methods for finding a particular element in the common solution set of a system of 
equilibrium problems and finitely many fixed point problems. This particular point is the unique solution of a variational inequality problem on the 
common solution set. The main idea of the paper is to combine three methods including the extragradient method, the Mann iteration method, the hybrid 
steepest-descent method with the parallel splitting-up technique to design the algorithms which improve the performance over some existing methods. The strongly convergent theorems are established under the widely used assumptions for 
equilibrium bifunctions.
\end{abstract}
\section{Introduction}\label{intro}
Let $C$ be a nonempty closed convex subset of a real Hilbert space $H$. Let $f:C\times C\to \Re$ be a bifunction. The equilibrium problem (EP) 
for $f$ on $C$ is to find $x^*\in C$ such that 
\begin{equation}\label{EP}
f(x^*,y)\ge 0,~\forall y\in C.
\end{equation}
The solution set of EP (\ref{EP}) is denoted by $EP(f,C)$. Mathematically, EP is a generalization of many mathematical models 
including variational inequality problems (VIP), optimization problems and fixed point problems (FPP), nonlinear and linear complemetarity problems, vector minimization problems 
and Nash equilibria problems, see for instance \cite{BO1994,FP2002,K2000,KS1980}. Due to this reason, EP has been recieved a lot of attention by many authors. Some 
notable methods for studying and solving EPs are the proximal point method \cite{K2000,R1976}, the splitting 
proximal method \cite{M2009}, the extragradient method and the Armijo linesearch method \cite{QMH2008}, the gradient-like projection method \cite{H2016}, the hybrid extragradient method \cite{A2013,HMA2016}, the extragradient - viscosity method 
\cite{VSN2013}. 

Let $S:C\to C$ be a mapping. Let us denote $Fix(S)$ by the fixed point set of $S$, i.e., $Fix(S)=\left\{x\in C:x=S(x)\right\}$. The problem of finding 
a common element of the fixed point set of a mapping and the solution set of an equilibrium problem is a task arising in various fields of applicable 
mathematics, sciences, engineering and economy, for example \cite{FP2002}. In \cite{HMA2016}, the authors presented a model which comes from 
Nash-Cournot model \cite{FP2002} for finding a point in the solution set $EP(f,C)\cap Fix(S)$. As a further extension, in this paper we consider 
the following common solution problem.
\begin{problem}\label{VIPoverEPandFPP} Find an element $x^*\in \Omega:=\left(\cap_{i\in I} EP(f_i,C)\right)\bigcap \left(\cap_{j\in J} Fix(S_j)\right)$, where
$f_i:C\times C\to \Re,~i\in I=\left\{1,2,\ldots,N\right\}$ are bifunctions and $S_j:C\to C,~j\in J=\left\{1,2,\ldots,M\right\}$ are mappings. 
\end{problem}
In recent years, the problem of finding a common solution of EPs and/or VIPs and FPPs has been widely and intensively studied by many authors, for example 
\cite{CH2005,ABH2014,AH2014,AH2014b,CGR2012,CGRS2012,H2015Korean,HMA2016,H2016b,H2015Iranian}. Problem \ref{VIPoverEPandFPP} 
includes many previously considered problems. 
When $S_j=I$ for all $j$, Problem \ref{VIPoverEPandFPP} becomes the problem of finding a common solution to EPs which was introduced and 
studied by Combettes and Hirstoaga in \cite{CH2005}. Using the resolvent of a bifunction, the authors proposed a general block-iterative algorithm 
for finding a common solution of EPs. A special case of this problem is the common solutions to variational inequalities problem (CSVIP) mentioned 
and analyzed intensively in \cite{CGRS2012} where the authors proposed an algorithm for CSVIP which combines the extragradient 
method \cite{K1976} with the hybrid (outer approximation) method. In a very recent work \cite{HMA2016}, Problem \ref{VIPoverEPandFPP} 
has been studied and analyzed in the case $M,~N>1$, the authors in \cite{HMA2016} proposed some parallel hybrid extragradient methods 
which combine the extended extragradient method \cite{QMH2008}, the Mann or Halpern iterations, the parallel splitting-up technique \cite{H2015Korean} 
and the outer approximation method (hybrid method). A notable problem in these algorithms 
is that at each iteration we must construct two closed convex subsets $C_n,~Q_n$ of the feasible set $C$ and compute the next approximation being the 
projection of the starting point $x_0$ on the intersection $C_n\cap Q_n$. These can be costly and affect the efficiency of the used method. 

On the other hand, for finding a particular solution of Problem \ref{VIPoverEPandFPP} when $M=N=1$, Maing$\rm \acute{e}$ and Moudafi \cite{MM2008} 
introduced the variational inequality problem: Find $x^*\in EP(f,C)\cap Fix(S)$ such that
\begin{equation}\label{VIP1}
\left\langle F(x^*),y-x^*\right\rangle\ge 0,~\forall y\in EP(f,C)\cap Fix(S),
\end{equation}
where $F:C\to H$ is $\eta$ - strongly 
monotone and $L$ - Lipschitz continuous, i.e., there two positive constants $\eta$ and $L$ such that, for all $x,~y\in C$,
$$ \left\langle F(x)-F(y),x-y\right\rangle \ge \eta ||x-y||^2, $$
$$ ||F(x)-F(y)||\le L ||x-y||. $$
Using the proximal point 
method for EP and the hybrid steepest - descent method introduced by Yamada and Ogura in \cite{Y2001}, 
Maing$\rm \acute{e}$ and Moudafi \cite{MM2008} proposed the following iterative method for VIP (\ref{VIP1}): Choose $x_0\in C$ and
\begin{equation}\label{MM2008}
\begin{cases}
z_n \in C~\mbox{such that}~f(z_n,y)+\frac{1}{r_n}\left\langle y-z_n,z_n-x_n\right\rangle\ge 0,~\forall y\in C,\\ 
x_{n+1}=(1-w)t_n+w St_n~\mbox{with}~t_n=z_n-\alpha_n Fz_n,
\end{cases}
\end{equation}
where $w$, $r_n$, $\alpha_n$ are suitable parameters. Recently, with the same idea, Vuong et al. \cite{VSN2013} have replaced the proximal point method by 
the extragradient method \cite{A2013,QMH2008} for computing $z_n$ in (\ref{MM2008}) and proposed the following extragradient - viscosity 
method for VIP (\ref{VIP1}): Choose $x_0\in C$ and 
\begin{equation}\label{VSN2013}
\begin{cases}
y_n = {\rm argmin} \{ \rho f(x_n, y) +\frac{1}{2}||x_n-y||^2:  y \in C\},\\
z_n = {\rm argmin} \{ \rho f(y_n, y) +\frac{1}{2}||x_n-y||^2:  y \in C\},\\
x_{n+1}=(1-w)t_n+w St_n~\mbox{with}~t_n=z_n-\alpha_n Fz_n,
\end{cases}
\end{equation}
where $w$, $\rho$, $\alpha_n$ are suitable parameters. The advantage of using the viscosity method is that it gives us strongly convergent algorithms 
which have more simple and elegant structures.

In this paper, motivated and inspired by the results in \cite{HMA2016,VSN2013,MM2008}, we propose two parallel 
algorithms for Problem \ref{VIPoverEPandFPP} which do not require constructing two set $C_n$, $Q_n$ and computing the projection onto their 
intersection per each iteration  as in \cite{HMA2016}. As the idea of Maing$\rm \acute{e}$ and Moudafi \cite{MM2008}, 
Vuong et al. \cite{VSN2013}, we also find a particular solution $x^*$ of Problem \ref{VIPoverEPandFPP} which satisfies the following variational 
inequality problem:
\begin{equation}\label{VIP2}
\left\langle F(x^*),y-x^*\right\rangle\ge 0,~\forall y\in \Omega,
\end{equation}
where $F:C\to H$ is $\eta$ - strongly monotone and $L$ - Lipschitz continuous. Let us denote $VIP(F,\Omega)$ by the solution set of VIP (\ref{VIP2}). 
Note that if $F(x)=x-u$ with $u$ being a suggested point in $H$ then VIP (\ref{VIP2}) reduces to the problem of finding an element 
$x^*\in \Omega$ which is the best approximation of $u$, i.e., $x^*=P_{\Omega}(u)$. 
 Firstly, using the extragradient method, we find semultaneously intermediate approximations for each equilibrium problems in the family. After that, among 
obtained approximations, the furthest one from the previous iterate is chosen. Based on this element, we compute in parallel other intermediate iterates for 
fixed point problems in this family. Similarly, we defined the next iterate and obtain the first algorithm. Next, as an improvement of 
finding furthest approximations in the first algorithm, we use convex combinations of component intermediate approximations and propose the second 
parallel algorithm. In our numerical experiments, with the first way, we see that the obtained algorithm 
seems to be more effective than the second one and hybrid methods proposed in \cite{HMA2016}. Some advantages of this performance in comparing with 
that of cyclic methods, specially when the numbers of subproblems $N,~M$ are large, can be found in \cite{AH2014,AH2014b,HMA2016,H2016a} 
and several references therein.

 This paper is organized as follows: In Sec. \ref{pre} we recall some definitions and preliminary results for the further use. Sec. \ref{main} 
deals with proposing the algorithms and proving their convergence. Finally, in Sec. \ref{example} we present a numerical example to illustrate the convergence 
of our algorithms and compare them with the parallel hybrid method in \cite{HMA2016}.
\section{Preliminaries} \label{pre}
Let $C$ be a nonempty closed convex subset of a real Hilbert space $H$. We begin with some definitions and properties of a demicontractive mapping. 
\begin{definition}\label{def1}
A mapping $S:C\to C$ is called: 
\begin{itemize}
\item [$\rm (i)$] \textit{nonexpansive} if $||S(x)-S(y)||\le ||x-y||$ for all $x,~y\in C$.
\item [$\rm (ii)$] \textit{quasi-nonexpansive} if $Fix(S)\ne\emptyset$ and 
$$ 
||S(x)-x^*||\le||x-x^*||,~\forall x^*\in Fix(S),~\forall x\in C.
 $$
\item [$\rm (iii)$] \textit{$\beta$ - demicontractive} if $Fix(S)\ne\emptyset$, and there exists $\beta\in [0,1)$ such that
$$ 
||S(x)-x^*||^2\le ||x-x^*||^2+\beta||x-S(x)||^2,~\forall x^*\in Fix(S),~\forall x\in C.
 $$
\item [$\rm (iv)$] \textit{demiclosed at zero} if, for each sequence $\left\{x_n\right\}\subset C$, $x_n\rightharpoonup x$, and $||S(x_n)-x_n||\to 0$ then 
$S(x)=x$.
\end{itemize}
\end{definition}
From the definitions above, we see that (i) $\Longrightarrow$ (ii) $\Longrightarrow$ (iii). It is well-known that 
each nonexpansive mapping is demiclosed at zero. Problem \ref{VIPoverEPandFPP} was considered in \cite{HMA2016} for nonexpansive mappings. 
In this paper, for more flexibility, we consider the mappings $S_j,~j\in J$ being demicontractive. We have the following result for a demicontractive mapping.
\begin{lemma}\cite[Remark 4.2]{M2008}\label{M2008a}
Assume that $S:C\to C$ be a $\beta$ - demicontractive mapping such that $Fix(S)\ne\emptyset$. 
Then
\begin{itemize}
\item [$\rm (i)$]  $S_w=(1-w)I+wS$ is a quasi-nonexpensive mapping over $C$ for every $w\in 
[0,1-\kappa]$. Furthermore
$$ ||S_wx-x^*||\le ||x-x^*||^2-w(1-\beta-w)||Sx-x||^2,~\forall x^*\in Fix(S),~\forall x\in C. $$
\item [$\rm (ii)$] $Fix(S)$ is closed and convex.
\end{itemize}
\end{lemma}
Next, we present some concepts of the monotonicity of a bifunction.
\begin{definition}\cite{BO1994} \label{def2} A bifunction $f:C\times C\to \Re$ is said to be
\begin{itemize}
\item [$\rm (i)$] \textit{monotone} on $C$ if 
$$ f(x,y)+f(y,x)\le 0,~\forall x,y\in C; $$
\item [$\rm (ii)$] \textit{pseudomonotone }on $C$ if 
$$ f(x,y)\ge 0 \Longrightarrow f(y,x)\le 0,~\forall x,y\in C;$$
\item [$\rm (iii)$] \textit{Lipschitz-type continuous} on $C$ if there exist two positive constants $c_1,c_2$ such that
$$ f(x,y) + f(y,z) \geq f(x,z) - c_1||x-y||^2 - c_2||y-z||^2, ~ \forall x,y,z \in C.$$
\end{itemize}
\end{definition}
We have the following result about the operator $F$ mentioned in Section \ref{intro}.
\begin{lemma}(cf. \cite[Lemma 3.1]{Y2001})\label{lemY2001}
Suppose that $F:C\to H$ is $\eta$ - strongly monotone and $L$ - Lipschitz continuous operator. 
By using arbitrarily fixed $\mu\in \left(0,\frac{2\eta}{L^2}\right)$. Define the mapping $G:C\to H$ by 
$$ G^\mu(x)=\left(I-\mu F\right)x,~x\in C. $$ 
Then
\begin{itemize}
\item [$\rm (i)$] $G^\mu$ is strictly contractive over $C$ with the contractive constant $\sqrt{1-\mu(2\eta-\mu L^2)}$.
\item [$\rm (ii)$] For all $\nu\in (0,\mu)$,
$$ ||G^{\nu}(y)-x||\le\left(1-\frac{\nu\tau}{\mu}\right)||y-x||+\nu ||F(x)||,$$
where $\tau=1-\sqrt{1-\mu(2\eta-\mu L^2)}\in (0,1)$.
\end{itemize}
\end{lemma}
\begin{proof}
(i) From the definition of $G^\mu$, the $\eta$ - strong monotonicity and $L$ - Lipschitz continuity of $F$, we obtain
\begin{eqnarray*}
||G^\mu(x)-G^\mu(y)||^2&=&||(x-y)-\mu(F(x)-F(y))||^2\\ 
&=&||x-y||^2-2\mu\left\langle x-y,F(x)-F(y)\right\rangle+\mu^2||F(x)-F(y)||^2\\
&\le& ||x-y||^2-2\mu\eta||x-y||^2+\mu^2L||x-y||^2\\
&=&(1-\mu(2\eta-\mu L))||x-y||^2.
\end{eqnarray*}
This yields the desired conclusion. Next, we prove claim (ii) in this lemma. From the defition of $G$, we have
\begin{eqnarray*}
||G^{\nu}(y)-x||&=&||\left(y-\nu F(y)\right)-\left(x-\nu F(x)\right)-\nu F(x)||\\ 
&\le&||\left(y-\nu F(y)\right)-\left(x-\nu F(x)\right)||+\nu ||F(x)||\\
&=&||\left(1-\frac{\nu}{\mu}\right)(y-x)+\frac{\nu}{\mu}\left[\left(y-\mu F(y)\right)-\left(x-\mu F(x)\right)\right]||+\nu ||F(x)||\\
&=&||\left(1-\frac{\nu}{\mu}\right)(y-x)+\frac{\nu}{\mu}\left[G^{\mu}(y)-G^{\mu}(x)\right]||+\nu ||F(x)||\\
&\le&\left(1-\frac{\nu}{\mu}\right)||y-x||+\frac{\nu}{\mu}\sqrt{1-\mu(2\eta-\mu L^2)}||y-x||+\nu ||F(x)||\\
&=&\left(1-\frac{\nu\tau}{\mu}\right)||y-x||+\nu ||F(x)||.
\end{eqnarray*}
Lemma \ref{lemY2001} is proved.
\end{proof}
Finally, we have the following technical lemma.
\begin{lemma}\cite[Remark 4.4]{M2008}\label{M2008}
Let $\left\{\epsilon_n\right\}$ be a sequence of non-negative real numbers. Suppose that for any 
integer $m$, there exists an integer $p$ such that $p\ge m$ and $\epsilon_p\le \epsilon_{p+1}$. Let $n_0$ 
be an integer such that $\epsilon_{n_0}\le \epsilon_{n_0+1}$ and define, for all integer $n\ge n_0$,
$$\tau(n)=\max\left\{k\in N:n_0\le k\le n, ~\epsilon_k\le \epsilon_{k+1}\right\}.$$
Then $0\le \epsilon_n \le \epsilon_{\tau(n)+1}$ for all $n\ge n_0$. Furthermore, the sequence 
$\left\{\tau(n)\right\}_{n\ge n_0}$ is non-decreasing and tends to $+\infty$ as $n\to\infty$.
\end{lemma}
\section{Main results}\label{main}
In this section, we propose two parallel algorithms for finding a solution of Problem \ref{VIPoverEPandFPP} and prove their convergence. The first algorithm 
is designed as follows.
\begin{algorithm}\label{Alg1} \textbf{Initialization.} Choose $x_0\in C$. The parameters $\rho,~\alpha_n,~\beta^j_n$ satisfy Condition 3 below.\\
\textbf{Step 1.} Find semultaneously approximations $y_n^i$, $i\in I$,
$$y_n^i = {\rm argmin} \{ \rho f_i(x_n, y) +\frac{1}{2}||x_n-y||^2:  y \in C\}.$$
\textbf{Step 2.} Find semultaneously approximations $z_n^i$, $i\in I$,
$$ z_n^i = {\rm argmin} \{ \rho f_i(y_n^i, y) +\frac{1}{2}||x_n-y||^2:  y \in C\}. $$
\textbf{Step 3.} Compute semultaneously approximations $u_n^j$, $j\in J$,
$$u_n^j=(1-\beta^j_n)t_n+\beta^j_n S_j t_n,$$
where $t_n=\bar{z}_n-\alpha_n F(\bar{z}_n)$ and $ \bar{z}_n = {\rm argmax}\{||z_n^i - x_n||: i\in I\}$.\\
\textbf{Step 4.} Pick $ x_{n+1} = {\rm argmax}\{||u_n^j - t_n||: j \in J\}.$ Set $n=n+1$ and go back \textbf{Step 1}.
\end{algorithm}
\begin{remark}
The intermediate approximation $\bar{z}_n$ in Step 3 of Algorithm \ref{Alg1} is the furthest element from $x_n$ among all ones $z_n^i,~i\in I$ 
and the next iterate $x_{n+1}$ in Step 4 is the furthest element from $t_n$ among all approximations $u_n^j,~j\in J$. 
\end{remark}
Throughout this paper, from the definitions of $\bar{z}_n$ and $x_{n+1}$ in Algorithm \ref{Alg1}, we denote $i_n\in I$ and $j_n\in J$ by the indices such that
$\bar{z}_n=z_n^{i_n}$ and $x_{n+1}=u_n^{j_n}$. For the sake of 
simplicity, we also write $\bar{y}_n:=y_n^{i_n}$.
In order to establish the convergence of Algorithm \ref{Alg1}, we install the following conditions for the bifunctions $f_i$, the mappings $S_j$ and 
the control parameters $\rho$, $\alpha_n$ and $\beta_n$.

\textbf{Condition 1}

\begin{itemize}
\item[\rm A1.]~$f_i$ is pseudomonotone on $C$ and $f(x,x)=0$ for all $x\in C$;
\item [\rm A2.]~$f_i$ is Lipschitz-type continuous on $C$ with the constants $c_1,c_2$;
\item [\rm A3.]~$\lim\sup_{n\to\infty}f_i(x_n,y)\le f(x,y)$ for each sequence $\left\{x_n\right\}$ converging weakly to $x$.
\item [\rm A4.]~$f_i(x,.)$ is convex and subdifferentiable on $C$ for every fixed $x\in C$.
\end{itemize}

\textbf{Condition 2}

\begin{itemize}
\item [\rm B1.]~$S_j$ is $\beta$ - demicontractive on $C$, where $\beta\in [0,1)$; 
\item [\rm B2.]~$S_j$ is demiclosed at zero.
\end{itemize}

\textbf{Condition 3}

\noindent(i)~ $0<\rho<\min\left\{\frac{1}{2c_1},\frac{1}{2c_2}\right\}$;\qquad (ii)~ $\lim\limits_{n\to\infty}\alpha_n=0,\sum\limits_{n=1}^\infty \alpha_n=+\infty$; \qquad
(iii) ~$0<a\le \beta_n^j<\frac{1-\beta}{2}$.

Hypothesis A2 was introduced by Mastroeni in \cite{M2003}. It is necessary to imply the convergence of  the auxiliary principle method for 
EPs. If $A:C\to H$ is a $L$ - Lipschitz continuous operator then the bifunction $f(x,y)=\left\langle A(x),y-x\right\rangle$ satisfies hypothesis A2.
It is easy to show that if $f_i$ satisfies conditions A1-A4 then $EP(f_i,C)$ is closed and convex (see, for instance \cite{QMH2008}). Under Condition 2, 
from Lemma \ref{M2008a}, $Fix(S_j)$ is closed and convex. Thus, $\Omega$ is also convex and closed. In this paper, we assume that $\Omega$ is nonempty. 
Hence, it follows from the assumptions of the operator $F$ that VIP (\ref{VIP2}) has the unique solution on $\Omega$, denoted by $x^*$.
We need the following lemmas.
\begin{lemma}\label{lem.A11}\cite{A2013,QMH2008}
Suppose that $\left\{x_n\right\},~ \left\{ y^i_n\right\},~ \left\{z^i_n\right\}$ are the sequences defined by Algorithm $\ref{Alg1}$. Then
\begin{itemize}
\item [$\rm (i)$] $\rho \left(f_i(x_n,y)-f_i(x_n,y_n^i)\right) \ge \left\langle y_n^i-x_n, y_n^i-y\right\rangle, \forall y\in C,~\forall i\in I.$
\item [$\rm (ii)$] 
$
||z_n^i-x^*||^2\le ||x_n-x^*||^2-(1-2\rho c_1)||y_n^i-x_n||^2-(1-2\rho c_2)||y_n^i-z_n^i||^2,~\forall i\in I .
$
\end{itemize}
\end{lemma}
\begin{lemma}\label{lem1}
Suppose that $\left\{x_n\right\},~ \left\{ y^i_n\right\},~ \left\{z^i_n\right\}$ are the sequences defined by Algorithm $\ref{Alg1}$. Then, for all $y\in C$,
\begin{eqnarray*}
\rho f_i(y^i_n,y)&\ge& \left\langle y_n^i-x_n, y_n^i-z_n^i\right\rangle -c_1\rho ||y_n^i-x_n||^2 -c_2\rho ||z_n^i-y_n^i||^2 + \left\langle z_n^i-x_n, z_n^i-y\right\rangle. 
\end{eqnarray*}
\end{lemma}
\begin{proof}
Substituting $y=z_n^i\in C$ into inequality (i) of Lemma \ref{lem.A11}, we obtain
\begin{equation}\label{eq:12a}
\rho \left(f_i(x_n,z_n^i)-f_i(x_n,y_n^i)\right) \ge \left\langle y_n^i-x_n, y_n^i-z_n^i\right\rangle.
\end{equation}
From the Lipschitz-type continuity of $f_i$ and the relation (\ref{eq:12a}), we have
\begin{eqnarray}
\rho f_i(y_n^i,z_n^i)&\ge&\rho \left( f_i(x_n,z_n^i)-f_i(x_n,y_n^i)\right)-c_1\rho ||y_n^i-x_n||^2 -c_2\rho ||z_n^i-y_n^i||^2 \nonumber\\
&\ge&\left\langle y_n^i-x_n, y_n^i-z_n^i\right\rangle -c_1\rho ||y_n^i-x_n||^2  -c_2\rho ||z_n^i-y_n^i||^2.\label{eq:12b}
\end{eqnarray}
Similarly to Lemma \ref{lem.A11}(i), from the definition of $z_n^i$, we obtain
$$ \rho \left(f_i(y^i_n,y)-f_i(y^i_n,z_n^i)\right) \ge \left\langle z_n^i-x_n, z_n^i-y\right\rangle, \forall y\in C.$$
Thus,
\begin{equation}\label{eq:12c}
\rho f_i(y^i_n,y) \ge \rho f_i(y^i_n,z_n^i)+ \left\langle z_n^i-x_n, z_n^i-y\right\rangle, \forall y\in C.
\end{equation}
Combining the relations (\ref{eq:12b}) and (\ref{eq:12c}), we obtain
\begin{eqnarray*}
\rho f_i(y^i_n,y)&\ge& \left\langle y_n^i-x_n, y_n^i-z_n^i\right\rangle -c_1\rho ||y_n^i-x_n||^2 -c_2\rho ||z_n^i-y_n^i||^2 + \left\langle z_n^i-x_n, z_n^i-y\right\rangle
\end{eqnarray*}
for all $y\in C$. Lemma \ref{lem1} is proved.
\end{proof}
\begin{lemma}\label{lem2}
Suppose that $\left\{x_n\right\},~ \left\{ \bar{y}_n\right\},~ \left\{\bar{z}_n\right\}$ are the sequences defined by Algorithm $\ref{Alg1}$. Then
\begin{eqnarray*}
||x_{n+1}-x^*||^2&\le&||x_n-x^*||^2-(1-2\rho c_1)||\bar{y}_n-x_n||^2-(1-2\rho c_2)||\bar{y}_n-\bar{z}_n||^2\\
&&-||x_{n+1}-\bar{z}_n||^2-2\alpha_n\left\langle x_{n+1}-x^*,F(\bar{z}_n)\right\rangle.
\end{eqnarray*}
\end{lemma}
\begin{proof}
Substituting $i=i_n$ into the second inequality of Lemma \ref{lem.A11}, we obtain 
\begin{equation}\label{eq:1*}
||\bar{z}_n-x^*||^2\le ||x_n-x^*||^2-(1-2\rho c_1)||\bar{y}_n-x_n||^2-(1-2\rho c_2)||\bar{y}_n-\bar{z}_n||^2.
\end{equation}
From the definitions of $x_{n+1}$ and $u_n^{j_n}$,
$$ ||x_{n+1}-t_n||^2=||u_n^{j_n}-t_n||^2=(\beta_n^j)^2||t_n-S_{j_n}t_n||^2 $$
which implies that
\begin{equation}\label{eq:2*}
||t_n-S_{j_n}t_n||^2=\frac{1}{(\beta_n^j)^2}||x_{n+1}-t_n||^2.
\end{equation}
Set $S_{j,\beta^j_n}=(1-\beta^j_n)I+\beta^j_nS_j$. From the definition of $x_{n+1}$, we have $x_{n+1}=S_{j_n,\beta^j_n}t_n$. Since $S_{j_n}$ is 
$\beta$ - demicontractive, it follows from Lemma \ref{M2008a} that $S_{j_n,\beta_n^{j_n}}$ is quasi-nonexpensive and 
\begin{eqnarray*}
||x_{n+1}-x^*||^2&=&||S_{j_n,\beta_n^{j_n}}t_n-x^*||^2\\
&\le&||t_n-x^*||^2-\beta_n^{j_n}(1-\beta-\beta_n^{j_n})||S_{j_n}t_n-t_n||^2\\
&=&||t_n-x^*||^2-\frac{1-\beta-\beta_n^{j_n}}{\beta_n^{j_n}}||x_{n+1}-t_n||^2
\end{eqnarray*}
in which the last equality is followed from the relation (\ref{eq:2*}). 
From the assumption of $\beta_n^{j_n}$, we see that $\frac{1-\beta-\beta_n^{j_n}}{\beta_n^{j_n}}\ge 1$. Thus, it follows from the last inequality that
\begin{equation}\label{eq:3**}
||x_{n+1}-x^*||^2\le||t_n-x^*||^2-||x_{n+1}-t_n||^2.
\end{equation}
From the definition of $t_n$, we have
\begin{eqnarray*}
||t_n-x^*||^2-||x_{n+1}-t_n||^2&=&||\bar{z}_n-\alpha_n F(\bar{z}_n)-x^*||^2-||x_{n+1}-(\bar{z}_n-\alpha_n F(\bar{z}_n))||^2\\
&=&||\bar{z}_n-x^*||^2-2\alpha_n \left\langle x_{n+1}-x^*,F(\bar{z}_n)\right\rangle-||x_{n+1}-\bar{z}_n||^2\\
&\le&||x_n-x^*||^2-(1-2\rho c_1)||\bar{y}_n-x_n||^2-(1-2\rho c_2)||\bar{y}_n-\bar{z}_n||^2\\
&&-2\alpha_n \left\langle x_{n+1}-x^*,F(\bar{z}_n)\right\rangle-||x_{n+1}-\bar{z}_n||^2
\end{eqnarray*}
in which the last inequality is followed from the relation (\ref{eq:1*}). The last inequality and the relation (\ref{eq:3**}) lead to the desired 
conclusion.
\end{proof}
\begin{lemma}\label{lem3}
The sequences $\left\{x_n\right\}$, $\left\{y^i_n\right\}$, $\left\{z^i_n\right\}$, $\left\{u^j_n\right\}$, $\left\{t_n\right\}$ are bounded for all $i\in I$ and $j\in J$.
\end{lemma}
\begin{proof}
For a fixed $\mu\in \left(0,\frac{2\eta}{L^2}\right)$. Since $\alpha_n\to 0$, we can assume that $\left\{\alpha_n\right\}\subset (0,\mu)$. From the definitions 
of $G^\mu$ in Lemma \ref{lemY2001} and of $t_n$ in Algorithm \ref{Alg1}, we have $t_{n}=G^{\alpha_{n}}(\bar{z}_{n})$. Using Lemma 
\ref{lemY2001}(ii) for $y=\bar{z}_{n}$, $x=x^*$ and $\nu=\alpha_n$, we obtain
\begin{equation}\label{eq:4*}
||t_n-x^*||=||G^{\alpha_{n}}(\bar{z}_{n})-x^*||\le\left(1-\frac{\alpha_{n}\tau}{\mu}\right)||\bar{z}_{n}-x^*||+\alpha_{n} ||F(x^*)||,
\end{equation}
where $\tau$ is defined as in Lemma \ref{lemY2001}. From the relation (\ref{eq:1*}) and the hypothesises of $\rho$, we obtain
\begin{equation}\label{eq:5*}
||\bar{z}_n-x^*||\le ||x_n-x^*||.
\end{equation}
From the relation (\ref{eq:3**}) with $n:=n-1$, we have 
\begin{equation*}
||x_{n}-x^*||^2\le||t_{n-1}-x^*||^2-||x_{n}-t_{n-1}||^2
\end{equation*}
which implies 
\begin{equation}\label{eq:6*}
||x_{n}-x^*||\le||t_{n-1}-x^*||.
\end{equation}
Thus, it follows from the relation (\ref{eq:5*}) that
\begin{equation*}
||\bar{z}_n-x^*||\le ||t_{n-1}-x^*||.
\end{equation*}
This together with (\ref{eq:4*}) implies that
\begin{eqnarray*}
||t_n-x^*||&\le&\left(1-\frac{\alpha_{n}\tau}{\mu}\right)||t_{n-1}-x^*||+\alpha_{n} ||F(x^*)||\\ 
&=&\left(1-\frac{\alpha_{n}\tau}{\mu}\right)||t_{n-1}-x^*||+\frac{\alpha_{n}\tau}{\mu}\left(\frac{\mu}{\tau}||F(x^*)||\right) \\
&\le&\max \left\{||t_{n-1}-x^*||,\frac{\mu}{\tau}||F(x^*)||\right\}.
\end{eqnarray*}
Thus 
$$ ||t_n-x^*||\le \max \left\{||t_{0}-x^*||,\frac{\mu}{\tau}||F(x^*)||\right\},~\forall n\ge 0.$$
This implies the boundedness of $\left\{t_n\right\}$. Hence, from (\ref{eq:5*}) and (\ref{eq:6*}), we see that the sequences $\left\{x_n\right\}$ and 
$\left\{\bar{z}_n\right\}$ are bounded. It follows from the definitions of $\bar{z}_n$ and $x_{n+1}$ that 
\begin{eqnarray*}
&&||z_n^i-x_n||\le||\bar{z}_n-x_n||,~\forall i\in I,\\ 
&& ||u_n^j-t_n||\le||x_{n+1}-t_n||,~\forall j\in J.
\end{eqnarray*}
Thus, the sequences $\left\{z^i_n\right\}$, $\left\{u^j_n\right\}$ are also bounded. Finally, the boundedness of $\left\{y^i_n\right\}$ is followed from 
Lemma \ref{lem.A11}(ii), the hypothesis of $\rho$ and the boundedness of the sequences $\left\{z^i_n\right\}$, $\left\{x_n\right\}$.
\end{proof}
\begin{theorem}\label{theo1}
Assume that Conditions 1, 2, 3 hold and the operator $F:C\to H$ is $\eta$ - 
strongly monotone and $L$ - Lipschitz continuous. In addition, the set $\Omega$ is nonempty. Then, the sequence $\left\{x_n\right\}$ generated by 
Algorithm \ref{Alg1} converges strongly to the unique solution $x^*$ of  VIP for $F$ on $\Omega$.
\end{theorem}
\begin{proof}
Since $\left\{x_n\right\}$, $\left\{\bar{z}_n\right\}$ are bounded and $F$ is $L$ - Lipschitz continuous, there exists a constant $K>0$ such that 
\begin{equation}\label{eq:7*}
2\left|\left\langle x_{n+1}-x^*,F(\bar{z}_n)\right\rangle\right|\le K.
\end{equation}
Set $\epsilon_n=||x_n-x^*||^2$. Using Lemma \ref{lem2} and the relation (\ref{eq:7*}), we obtain 
\begin{equation}\label{eq:8}
\epsilon_{n+1}-\epsilon_n+(1-2\rho c_1)||\bar{y}_n-x_n||^2+(1-2\rho c_2)||\bar{y}_n-\bar{z}_n||^2+||x_{n+1}-\bar{z}_n||^2\le\alpha_n K.
\end{equation}
We consider two cases.\\
\textbf{Case 1.} There exists $n_0$ such that $\left\{\epsilon_n\right\}$ is decreasing for all $n\ge n_0$. Thus, from $\epsilon_n\ge 0$ for all $n\ge 0$, 
there exists the limit of $\left\{\epsilon_n\right\}$, i.e., $\epsilon_n\to \epsilon\ge 0$ as $n\to\infty$. Hence, it follows from (\ref{eq:8}), the hypothesis
of $\rho$ and $\alpha_n\to 0$ that
\begin{equation}\label{eq:9}
||\bar{y}_n-x_n||\to 0,~||\bar{y}_n-\bar{z}_n||\to 0,~||x_{n+1}-\bar{z}_n||\to 0.
\end{equation}
From the relation (\ref{eq:9}) and the triangle inequality, we obtain 
\begin{equation}\label{eq:10}
||x_{n+1}-x_n||\to 0,~||\bar{z}_n-x_n||\to 0.
\end{equation}
From the definition of $\bar{z}_n$, we obtain $||z_n^i-x_n||\le ||\bar{z}_n-x_n||,~\forall i\in I$. This together with (\ref{eq:10}) implies that
\begin{equation}\label{eq:11}
||z^i_n-x_n||\to 0,~\forall i\in I.
\end{equation}
From Lemma \ref{lem.A11}(ii) and the triangle inequality,
\begin{eqnarray*}
(1-2\rho c_1)||y_n^i-x_n||^2&+&(1-2\rho c_2)||y_n^i-z_n^i||^2\le ||x_n-x^*||^2-||z_n^i-x^*||^2\\ 
&\le& \left(||x_n-x^*||-||z_n^i-x^*||\right)\left(||x_n-x^*||+||z_n^i-x^*||\right)\\
&\le&||x_n-z_n^i||\left(||x_n-x^*||+||z_n^i-x^*||\right).
\end{eqnarray*}
Passing to the limit in the last inequality and using the hypothesis of $\rho$, the boundedness of $\left\{x_n\right\},\left\{z_n^i\right\}$ and (\ref{eq:11}), 
we obtain
\begin{equation}\label{eq:12}
||y_n^i-x_n||\to 0,~||y_n^i-z_n^i||\to 0,~\forall i\in I.
\end{equation}
Since $\left\{\bar{z}_n\right\}$ is bounded, without loss of generality, we can assume that there exists a subsequence $\left\{\bar{z}_m\right\}$ of 
$\left\{\bar{z}_n\right\}$ converging weakly to $p$ such that
\begin{equation}\label{eq:15}
\lim_{n\to\infty}\inf \left\langle \bar{z}_n-x^*, Fx^*\right\rangle=\lim_{m\to\infty}\left\langle \bar{z}_m-x^*,Fx^*\right\rangle.
\end{equation}
Now, we prove that 
$p\in \Omega$. Indeed, it follows from Lemma \ref{lem1} that, for all $y\in C$,
\begin{eqnarray*}
\rho f_i(y^i_m,y)&\ge& \left\langle y_m^i-x_m, y_m^i-z_m^i\right\rangle -c_1\rho ||y_m^i-x_m||^2 -c_2\rho ||z_m^i-y_m^i||^2 + \left\langle z_m^i-x_m, z_m^i-y\right\rangle. 
\end{eqnarray*}
From $\bar{z}_n\rightharpoonup p$ and the relations (\ref{eq:10}) and (\ref{eq:12}), we obtain $x_n\rightharpoonup p$, $y^i_n\rightharpoonup p$, 
$z^i_n\rightharpoonup p$. Thus, letting $m\to \infty$ in the last inequality and using hypothesis A3, $\rho>0$ and (\ref{eq:12}), we obtain 
$$ 0\le \lim\sup_{m\to\infty}f_i(y_n,y)\le f_i(p,y),~\forall y\in C,~\forall i\in I. $$
Thus, $p\in \cap_{i\in I}EP(f_i,C)$. 
Moreover, since $u_m^j=(1-\beta_m^j)t_m+\beta_m^jS_jt_m$ and $\beta_m^j\ge a>0$,
\begin{equation}\label{eq:13}
||t_m-S_jt_m||=\frac{1}{\beta_m^j}||u_m^j-t_m||\le \frac{1}{a}||u_m^j-t_m||\le \frac{1}{a}||x_{m+1}-t_m||,
\end{equation}
in which the last inequality is followed from the definition of $x_{m+1}$. From the definition of $t_m$, $\alpha_m\to 0$ and the boundedness of 
$\left\{\bar{z}_m\right\}$, we obtain 
\begin{equation}\label{eq:14}
||t_m-\bar{z}_m||=\alpha_m||F(\bar{z}_m)||\to 0.
\end{equation}
This together with (\ref{eq:9}) implies that $||x_{m+1}-t_m||\to 0$. Thus, it follows from (\ref{eq:13}) that 
$||t_m-S_jt_m||\to 0$ and $t_m\rightharpoonup p$. Since $S_j$ is demiclosed at zero, $p\in \cap_{j\in J}Fix(S_j)$. Hence, $p\in \Omega$.

In order to finish Case 1, we show that 
$$\epsilon_n=||x_n-x^*||^2\to \epsilon=0.$$
Since $||x_n-\bar{z}_n||\to 0$, $||\bar{z}_n-x^*||^2\to \epsilon$. From (\ref{eq:15}), $\bar{z}_n\rightharpoonup p\in \Omega$ and $x^*\in VIP(F,\Omega)$, 
one has
\begin{equation}\label{eq:16}
\lim_{n\to\infty}\inf \left\langle \bar{z}_n-x^*, Fx^*\right\rangle=\lim_{m\to\infty}\left\langle \bar{z}_m-x^*,Fx^*\right\rangle=\left\langle p-x^*,Fx^*\right\rangle\ge 0.
\end{equation}
From the $\eta$ - strongly monotonicity of $F$,
\begin{eqnarray*}
\left\langle x_{n+1}-x^*, F\bar{z}_n\right\rangle&=&\left\langle x_{n+1}-\bar{z}_n, F\bar{z}_n\right\rangle+\left\langle\bar{z}_n-x^*, F\bar{z}_n\right\rangle\\ 
&=& \left\langle x_{n+1}-\bar{z}_n, F\bar{z}_n\right\rangle+\left\langle\bar{z}_n-x^*, F\bar{z}_n-Fx^*\right\rangle+\left\langle\bar{z}_n-x^*,Fx^*\right\rangle\\ 
&\ge&\left\langle x_{n+1}-\bar{z}_n, F\bar{z}_n\right\rangle+\eta||\bar{z}_n-x^*||^2+\left\langle\bar{z}_n-x^*,Fx^*\right\rangle.
\end{eqnarray*}
This together with $||x_{n+1}-\bar{z}_n||\to 0$, $||\bar{z}_n-x^*||^2\to \epsilon$ and (\ref{eq:16}) implies that
\begin{equation}\label{eq:17}
\lim_{n\to\infty}\inf \left\langle x_{n+1}-x^*, F\bar{z}_n\right\rangle\ge \eta\epsilon.
\end{equation}
Assume that $\epsilon>0$, then there exists a positive integer $n_0$ such that 
\begin{equation}\label{eq:18}
\left\langle x_{n+1}-x^*, F\bar{z}_n\right\rangle\ge \frac{1}{2}\eta\epsilon,~\forall n\ge n_0.
\end{equation}
It follows from Lemma \ref{lem2} that 
\begin{equation}\label{eq:19}
||x_{n+1}-x^*||^2\le||x_n-x^*||^2-2\alpha_n\left\langle x_{n+1}-x^*,F(\bar{z}_n)\right\rangle.
\end{equation}
Combining (\ref{eq:18}) and (\ref{eq:19}), we obtain 
\begin{equation*}
||x_{n+1}-x^*||^2-||x_n-x^*||^2\le -\alpha_n\eta\epsilon,~\forall n\ge n_0,
\end{equation*}
or 
\begin{equation*}
\epsilon_{n+1}-\epsilon_n\le -\alpha_n\eta\epsilon,~\forall n\ge n_0.
\end{equation*}
Thus,
\begin{equation}\label{eq:20}
\epsilon_{n+1}-\epsilon_{n_0}\le -\eta\epsilon \sum_{k=n_0}^{n+1}\alpha_k.
\end{equation}
Since $\eta>0$, $\epsilon>0$ and $\sum_{n=1}^{\infty}\alpha_n=+\infty$, it follows from (\ref{eq:20}) that $\epsilon_n\to -\infty$. This is contradiction. 
Therefore $\epsilon=0$ or $x_n\to x^*$.\\
\textbf{Case 2.} There exists a subsequence $\left\{\epsilon_{n_i}\right\}$ of $\left\{x_n\right\}$ such that $\epsilon_{n_i}\le \epsilon_{n_i+1}$ for all $i\ge 0$. \\
It follows from Lemma \ref{M2008} that 
\begin{equation}\label{eq:21}
\epsilon_{\tau(n)}\le \epsilon_{\tau(n)+1},~\epsilon_n \le \epsilon_{\tau(n)+1},~\forall n\ge n_0.
\end{equation}
where $\tau(n)=\max\left\{k\in N:n_0\le k\le n, ~\epsilon_k\le \epsilon_{k+1}\right\}$. Furthermore, the sequence 
$\left\{\tau(n)\right\}_{n\ge n_0}$ is non-decreasing and $\tau(n)\to +\infty$ as $n\to\infty$.

It follows from (\ref{eq:8}), the hypothesises 
of $\rho$, $\epsilon_{\tau(n)}\le \epsilon_{\tau(n)+1}$ and $\alpha_{\tau(n)}\to 0$ that
\begin{equation}\label{eq:22}
||\bar{y}_{\tau(n)}-x_{\tau(n)}||\to 0,~||\bar{y}_{\tau(n)}-\bar{z}_{\tau(n)}||\to 0,~||x_{{\tau(n)}+1}-\bar{z}_{\tau(n)}||\to 0.
\end{equation}
These together with the triangle inequality imply that $||x_{{\tau(n)}}-\bar{z}_{\tau(n)}||\to 0.$ Thus, from the definition of the index $i_{\tau(n)}$, we 
have 
\begin{equation}\label{eq:22*}
||x_{\tau(n)}-{z}^i_{\tau(n)}||\to 0,~\forall i\in I.
\end{equation}
From Lemma \ref{lem.A11}(ii) and the triangle inequality,
\begin{eqnarray*}
(1-2\rho c_1)||y_{\tau(n)}^i-x_{\tau(n)}||^2&+&(1-2\rho c_2)||y_{\tau(n)}^i-z_{\tau(n)}^i||^2\le ||x_{\tau(n)}-x^*||^2-||z_{\tau(n)}^i-x^*||^2\\ 
&\le& \left(||x_{\tau(n)}-x^*||-||z_{\tau(n)}^i-x^*||\right)\left(||x_{\tau(n)}-x^*||+||z_{\tau(n)}^i-x^*||\right)\\
&\le&||x_{\tau(n)}-z_{\tau(n)}^i||\left(||x_{\tau(n)}-x^*||+||z_{\tau(n)}^i-x^*||\right).
\end{eqnarray*}
Passing to the limit in the last inequality and using the hypothesis of $\rho$, the boundedness of $\left\{x_{\tau(n)}\right\},\left\{z_{\tau(n)}^i\right\}$ 
and (\ref{eq:22*}), 
we obtain
\begin{equation}\label{eq:22**}
||y_{\tau(n)}^i-x_{\tau(n)}||\to 0,~||y_{\tau(n)}^i-z_{\tau(n)}^i||\to 0,~\forall i\in I.
\end{equation}
Since $\left\{\bar{z}_{\tau(n)}\right\}$ is bounded, there exists a subsequence 
$\left\{\bar{z}_{\tau(n_k)}\right\}$ of $\left\{\bar{z}_{\tau(n)}\right\}$ converging weakly to $p$ such that 
\begin{equation}\label{eq:22***}
\lim\inf_{n\to\infty}\left\langle \bar{z}_{\tau(n)}-x^*,F(x^*)\right\rangle=\lim_{k\to\infty}\left\langle \bar{z}_{\tau(n_k)}-x^*,F(x^*)\right\rangle
\end{equation}
From (\ref{eq:22}), (\ref{eq:22**}) and $\bar{z}_{\tau(n_k)}\rightharpoonup p$, we also have $x_{\tau(n_k)}\rightharpoonup p$, 
$y^i_{\tau(n_k)}\rightharpoonup p$, $z^i_{\tau(n_k)}\rightharpoonup p$. Now, we show that $p\in \Omega$. Indeed, it follows from Lemma 
\ref{lem1} that, for all $y\in C$,
\begin{eqnarray*}
\rho f_i(y^i_{\tau(n_k)},y)&\ge& \left\langle y_{\tau(n_k)}^i-x_{\tau(n_k)}, y_{\tau(n_k)}^i-z_{\tau(n_k)}^i\right\rangle -c_1\rho ||y_{\tau(n_k)}^i-x_{\tau(n_k)}||^2 \\
&&-c_2\rho ||z_{\tau(n_k)}^i-y_{\tau(n_k)}^i||^2 + \left\langle z_{\tau(n_k)}^i-x_{\tau(n_k)}, z_{\tau(n_k)}^i-y\right\rangle. 
\end{eqnarray*}
Passing to the limit in the last inequality as $k\to\infty$ and using (\ref{eq:22*}), (\ref{eq:22**}), $\rho>0$ and A3, we obtain 
$$ 0\le \lim\sup_{k\to\infty} f_i(y^i_{\tau(n_k)},y)\le f_i(p,y),~\forall y\in C,~\forall i\in I.$$
Thus, $p\in \cap_{i\in I}EP(f_i,C)$. From $u_{\tau(n_k)}^j=(1-\beta^j_{\tau(n_k)})t_{\tau(n_k)}+\beta^j_{\tau(n_k)}S_jt_{\tau(n_k)}$ 
and $\beta^j_{\tau(n_k)}\ge a>0$, we see that
\begin{equation}\label{eq:23*}
||t_{\tau(n_k)}-S_jt_{\tau(n_k)}||=\frac{1}{\beta^j_{\tau(n_k)}}||u_{\tau(n_k)}^j-t_{\tau(n_k)}||\le \frac{1}{a}||u_{\tau(n_k)}^j-t_{\tau(n_k)}||\le \frac{1}{a}||x_{{\tau(n_k)}+1}-t_{\tau(n_k)}||,
\end{equation}
in which the last inequality is followed from the definition of $x_{\tau(n_k)+1}$. It follows from the definition of $t_{\tau(n_k)}$, $\alpha_{\tau(n_k)}\to 0$ 
and the boundedness of $\left\{\bar{z}_{\tau(n_k)}\right\}$ that
\begin{equation}\label{eq:23**}
||t_{\tau(n_k)}-\bar{z}_{\tau(n_k)}||=\alpha_{\tau(n_k)}||F(\bar{z}_{\tau(n_k)})||\to 0.
\end{equation}
This together with (\ref{eq:22}) implies that $||x_{{\tau(n_k)}+1}-t_{\tau(n_k)}||\to 0$. Thus, from (\ref{eq:23*}) and $x_{{\tau(n_k)}+1}\rightharpoonup p$, 
we obtain  $||t_{\tau(n_k)}-S_jt_{\tau(n_k)}||\to 0$ and $t_{\tau(n_k)}\rightharpoonup p$. Since $S_j$ is demiclosed at zero, $p\in \cap_{j\in J}Fix(S_j)$. 
Hence, $p\in \Omega$.

Now, we prove that $x_{\tau(n_k)}\to x^*$. It follows from Lemma \ref{lem2} that 
\begin{eqnarray*}
2\alpha_{\tau(n)}\left\langle x_{\tau(n)+1}-x^*,F(\bar{z}_{\tau(n)})\right\rangle&\le&\epsilon_{\tau(n)}-\epsilon_{\tau(n)+1}-(1-2\rho c_1)||
\bar{y}_{\tau(n)}-x_{\tau(n)}||^2\\
&&-(1-2\rho c_2)||\bar{y}_{\tau(n)}-\bar{z}_{\tau(n)}||^2-||x_{\tau(n)+1}-\bar{z}_{\tau(n)}||^2.
\end{eqnarray*}
Thus, 
\begin{equation}\label{eq:23}
\left\langle x_{\tau(n)+1}-x^*,F(\bar{z}_{\tau(n)})\right\rangle \le 0
\end{equation}
because of $\alpha_{\tau(n)}>0$, $\epsilon_{\tau(n)}\le \epsilon_{\tau(n)+1}$ and the hypothesis of $\rho$. From the $\eta$ - strong 
monotonicity and the relation (\ref{eq:23}),
\begin{eqnarray*}
\eta||\bar{z}_{\tau(n)}-x^*||^2&\le&\left\langle \bar{z}_{\tau(n)}-x^*, F\bar{z}_{\tau(n)}-Fx^*\right\rangle\\ 
&=& \left\langle \bar{z}_{\tau(n)}-x^*, F\bar{z}_{\tau(n)}\right\rangle-\left\langle \bar{z}_{\tau(n)}-x^*,Fx^*\right\rangle\\
&=& \left\langle \bar{z}_{\tau(n)}-x_{\tau(n)+1}, F\bar{z}_{\tau(n)}\right\rangle+\left\langle x_{\tau(n)+1}-x^*, F\bar{z}_{\tau(n)}\right\rangle-\left\langle \bar{z}_{\tau(n_k)}-x^*,Fx^*\right\rangle\\
&\le&\left\langle \bar{z}_{\tau(n)}-x_{\tau(n)+1}, F\bar{z}_{\tau(n)}\right\rangle-\left\langle \bar{z}_{\tau(n)}-x^*,Fx^*\right\rangle.
\end{eqnarray*}
This together with (\ref{eq:22}), (\ref{eq:22***}) and $\bar{z}_{\tau(n_k)}\rightharpoonup p$ implies that
\begin{eqnarray*}
\lim\sup_{n\to\infty}\eta||\bar{z}_{\tau(n)}-x^*||^2&\le&-\lim\inf_{n\to\infty}\left\langle \bar{z}_{\tau(n)}-x^*,Fx^*\right\rangle.\\ 
&=&-\lim_{k\to\infty}\left\langle \bar{z}_{\tau(n_k)}-x^*,Fx^*\right\rangle.\\ 
&=&-\lim_{k\to\infty}\left\langle p-x^*,Fx^*\right\rangle\le 0,
\end{eqnarray*}
in which the last inequality is followed from $p\in \Omega$ and $x^*\in VIP(F,\Omega)$. Thus 
\begin{equation*}
\lim_{n\to\infty} ||\bar{z}_{\tau(n)}-x^*||^2=0
\end{equation*}
because of $\eta>0$. This together with (\ref{eq:22}) implies 
that $\lim_{k\to\infty} ||x_{\tau(n)+1}-x^*||^2=0$. Thus, $\epsilon_{\tau(n)+1}\to 0$. It follows from (\ref{eq:21}) that 
$0\le\epsilon_n\le \epsilon_{\tau(n)+1}\to 0$. Hence, $\epsilon_n\to 0$ or $x_n\to x^*$ as $n\to\infty$. Theorem \ref{theo1} is proved.

\end{proof}
Next, by replacing the element $\bar{z}_n$ in Step 3 and the next one $x_{n+1}$ in Step 4 of Algorithm \ref{Alg1} by convex combinations of $z_n^i,~i\in I$ 
and of $u_n^j,~j\in I$, respectively, we come to the following algorithm. 
\begin{algorithm}\label{Alg2} \textbf{Initialization.} Choose $x_0\in C$. The parameters $\rho,~\alpha_n,~\beta^j_n,~w_n^i,~\gamma_n^j$ satisfy 
Condition 4 below.\\
\textbf{Step 1.} Find semultaneously approximations $y_n^i$, $i\in I$
$$y_n^i = {\rm argmin} \{ \rho f_i(x_n, y) +\frac{1}{2}||x_n-y||^2:  y \in C\}.$$
\textbf{Step 2.} Find semultaneously approximations $z_n^i$, $i\in I$
$$ z_n^i = {\rm argmin} \{ \rho f_i(y_n^i, y) +\frac{1}{2}||x_n-y||^2:  y \in C\}. $$
\textbf{Step 3.} Compute 
$$ z_n=\sum_{i\in I}w_n^i z_n^i, $$
$$ x_{n+1}=\sum\limits_{j\in J}\gamma_n^j\left[(1-\beta_n^j)t_n+\beta_n^j S_j t_n\right], $$
where $t_n={z}_n-\alpha_n F({z}_n)$. Set $n=n+1$ and go back \textbf{Step 1}.
\end{algorithm}
From Step 3 of Algorithm \ref{Alg2}, we see that the problems of computing $z_n$ and $x_{n+1}$ are more simpler than those of computing 
$\bar{z}_n$ and $x_{n+1}$ in Steps 3, 4 of Algorithm \ref{Alg1}. This is also illustrated in our numerical experiments in Sec. \ref{example} 
where time for execution of this algorithm is less consuming than Algorithm \ref{Alg1} and the parallel hybrid extragradient method in \cite{HMA2016}. 
In order to obtain the convergence of Algorithm \ref{Alg2}, we install the following condition on the control parameters in Algorithm \ref{Alg2}.\\

\textbf{Condition 4}\qquad Condition 3 holds and
\begin{itemize}
\item [$\rm (iv)$] $w_n^i\in (0,1)$, $\sum\limits_{i\in I}w_n^i=1$, $\lim\limits_{n}\inf w^i_n>0$ for all $i\in I$ and $n\ge 0$.
\item [$\rm (v)$] $\gamma_n^j\in (0,1)$, $\sum\limits_{j\in J}\gamma_n^j=1$, $\lim\limits_{n}\inf\gamma^j_n>0$ for all $j\in J$ and $n\ge 0$.
\end{itemize}
\begin{theorem}\label{theo2}
The concusion of Theorem \ref{theo1} remains true for Algorithm \ref{Alg2} under Conditions 1, 2 and 4.
\end{theorem}
\begin{proof} We divide the proof of Theorem \ref{theo2} into several steps.\\
 \textbf{Claim 1.}
Suppose that $x^*\in VIP(F,\Omega)$. Then 
\begin{eqnarray*}
||x_{n+1}-x^*||^2&\le&||x_n-x^*||^2-(1-2\rho c_1)\sum_{i\in I}w_n^i ||{y}^i_n-x_n||^2-(1-2\rho c_2)\sum_{i\in I}w_n^i ||{y}^i_n-{z}^i_n||^2\\
&&-||x_{n+1}-{z}_n||^2-2\alpha_n\left\langle x_{n+1}-x^*,F({z}_n)\right\rangle.
\end{eqnarray*}
\textit{The proof of Claim 1.} 
From the convexity of $||.||^2$ and Lemma \ref{lem.A11}, we obtain 
\begin{eqnarray}
||{z}_n-x^*||^2&=& \sum_{i\in I}||w_n^i(z_n^i-x^*)||^2\le \sum_{i\in I}w_n^i||z_n^i-x^*||^2\le||x_n-x^*||^2\nonumber\\
&&-(1-2\rho c_1)\sum_{i\in I}w_n^i||{y}^i_n-x_n||^2-(1-2\rho c_2)\sum_{i\in I}w_n^i||{y}^i_n-\bar{z}_n||^2.\label{eq:25}
\end{eqnarray}
Setting $u_n^j:=S_{j,\beta_n}t_n=(1-\beta_n^j)t_n+\beta^j_nS_jt_n$. From the definitions of $x_{n+1}$ and $u_n^j$, we have 
$x_{n+1}=\sum\limits_{j\in J}\gamma_n^ju_n^j$ and $||u_n^{j}-t_n||^2=(\beta_n^j)^2||t_n-S_{j}t_n||^2$. Thus, by the convexity of $||.||^2$, 

\begin{eqnarray}
||x_{n+1}-t_n||^2&=&||\sum\limits_{j\in J}\gamma_n^j(u_n^j-t_n)||^2\le \sum\limits_{j\in J}\gamma_n^j||u_n^j-t_n)||^2=\sum\limits_{j\in J}\gamma_n^j(\beta_n^j)^2||t_n-S_{j}t_n||^2.\label{eq:26}
\end{eqnarray}
From the hypothesis of $\beta_n^j$,
\begin{equation}\label{eq:26*}
\frac{1-\beta-\beta^j_n}{\beta_n^j}\ge 1.
\end{equation}
By the convexity of $||.||^2$, Lemma \ref{M2008a}(i), $\sum\limits_{j\in J}\gamma_k^j=1$ and the relations (\ref{eq:26}), (\ref{eq:26*}), we obtain
\begin{eqnarray}
||x_{n+1}-x^*||^2&=&||\sum\limits_{j\in J}\gamma_n^j(u_n^j-x^*)||^2\le \sum\limits_{j\in J}\gamma_n^j||u_n^j-x^*||^2=\sum\limits_{j\in J}\gamma_n^j||S_{j,\beta_n}t_n-x^*||^2\nonumber\\
&\le&\sum\limits_{j\in J}\gamma_n^j\left[||t_n-x^*||^2-\beta_n^j(1-\beta-\beta_n^j)||S_{j}t_n-t_n||^2\right]\nonumber\\
&=&||t_n-x^*||^2-\sum\limits_{j\in J}\beta_n^j(1-\beta-\beta_n^j)\gamma_n^j ||S_{j}t_n-t_n||^2 \nonumber\\
&=&||t_n-x^*||^2-\sum\limits_{j\in J}\frac{1-\beta-\beta^j_n}{\beta_n^j}(\beta_n^j)^2\gamma_n^j ||S_{j}t_n-t_n||^2 \nonumber\\
&\le&||t_n-x^*||^2-\sum\limits_{j\in J}(\beta_n^j)^2\gamma_n^j ||S_{j}t_n-t_n||^2 \label{eq:27}\\
&\le&||t_n-x^*||^2-||x_{n+1}-t_n||^2\nonumber.
\end{eqnarray}
Thus
\begin{equation}\label{eq:28}
||x_{n+1}-x^*||^2\le||t_n-x^*||^2-||x_{n+1}-t_n||^2\le ||t_n-x^*||^2.
\end{equation}
This together with (\ref{eq:25}), the definition of $t_n$ implies that
\begin{eqnarray*}
||x_{n+1}-x^*||^2&\le&||t_n-x^*||^2-||x_{n+1}-t_n||^2\\
&=&||{z}_n-\alpha_n F({z}_n)-x^*||^2-||x_{n+1}-({z}_n-\alpha_n F({z}_n))||^2\\
&=&||{z}_n-x^*||^2-2\alpha_n \left\langle x_{n+1}-x^*,F({z}_n)\right\rangle-||x_{n+1}-{z}_n||^2\\
&\le&||x_n-x^*||^2-(1-2\rho c_1)\sum_{i\in I}w_n^i||{y}^i_n-x_n||^2-(1-2\rho c_2)\sum_{i\in I}w_n^i||{y}^i_n-{z}^i_n||^2\\
&&-2\alpha_n \left\langle x_{n+1}-x^*,F({z}_n)\right\rangle-||x_{n+1}-{z}_n||^2
\end{eqnarray*}
\textbf{Claim 2.} The sequences $\left\{x_n\right\}$, $\left\{y^i_n\right\}$, $\left\{z^i_n\right\}$, $\left\{t_n\right\}$ are 
bounded for all $i\in I$ and $j\in J$.\\
\textit{The proof of Claim 2.} Repeating the proof of Lemma \ref{lem3}, we can conclude that $\left\{t_n\right\}$ is a bounded sequence. It 
follows from (\ref{eq:28}) that $\left\{x_n\right\}$ is bounded. The boundedness of $\left\{y^i_n\right\}$, $\left\{z^i_n\right\}$ is followed from 
Lemma \ref{lem.A11}(ii).\\
\textbf{Claim 3.} If $\left\{x_m\right\}$ is some subsequence of $\left\{x_n\right\}$ such that $||x_{m+1}-t_m||\to 0$ then $||S_{j}t_m-t_m||\to 0$ for all $j\in J$.\\
\textit{The proof of Claim 3.} From $\beta^j_m\ge a>0$, the relation (\ref{eq:27}) and the triangle inequality,
\begin{eqnarray*}
a^2\sum\limits_{j\in J}\gamma_m^j ||S_{j}t_m-t_m||^2&\le&\sum\limits_{j\in J}(\beta_m^j)^2\gamma_m^j ||S_{j}t_m-t_m||^2\\
&\le&||t_m-x^*||^2-||x_{m+1}-x^*||^2\\ 
&=& \left(||t_m-x^*||-||x_{m+1}-x^*||\right)\left(||t_m-x^*||+||x_{m+1}-x^*||\right)\\
&\le&||t_m-x_{m+1}||\left(||t_m-x^*||+||x_{m+1}-x^*||\right)
\end{eqnarray*}
Passing to the limit in the last inequality as $m\to\infty$ and using the hypothesis $||t_m-x_{m+1}||\to 0$, the boundedness of $\left\{x_m\right\}$, 
$\left\{t_m\right\}$, we obtain 
$$ \sum\limits_{j\in J}\gamma_m^j ||S_{j}t_m-t_m||^2\to 0. $$
This together with the hypothesis $\lim\inf_{n} \gamma_n^j>0$ yields the desired conclusion.\\
\textbf{Claim 4.} $x_n\to x^*$ as $n\to\infty$, where $x^*$ is the unique solution of VIP (\ref{VIP2}).\\
\textit{The proof of Claim 4}. 
Since $\left\{x_n\right\}$, $\left\{{z}_n\right\}$ are bounded and $F$ is $L$ - Lipschitz continuous, there exists a constant $K>0$ such that 
\begin{equation}\label{eq:7a}
2\left|\left\langle x_{n+1}-x^*,F({z}_n)\right\rangle\right|\le K.
\end{equation}
Set $\epsilon_n=||x_n-x^*||^2$. Using Lemma \ref{lem2}, we obtain 
\begin{equation}\label{eq:8a}
\epsilon_{n+1}-\epsilon_n+(1-2\rho c_1)\sum_{i\in I}w_n^i||{y}^i_n-x_n||^2+(1-2\rho c_2)\sum_{i\in I}w_n^i||{y}^i_n-{z}^i_n||^2+||x_{n+1}-{z}_n||^2\le\alpha_n K.
\end{equation}
We consider two cases.\\
\textbf{Case 1.} There exists $n_0$ such that $\left\{\epsilon_n\right\}$ is decreasing for all $n\ge n_0$. Since $\epsilon_n\ge 0$ for all $n\ge 0$, 
there exists the limit of $\left\{\epsilon_n\right\}$, i.e., $\epsilon_n\to \epsilon$ as $n\to\infty$. Thus, it follows from (\ref{eq:8a}), the facts $1-2\rho c_1>0$, 
$1-2\rho c_2>0$, $\lim\inf_n w_n^i>0$ and $\alpha_n\to 0$ that
\begin{equation}\label{eq:9a}
||y_n^i-x_n||\to 0,~||y_n^i-z_n^i||\to 0,~||x_{n+1}-{z}_n||\to 0,~\forall i\in I.
\end{equation}
Using (\ref{eq:9a}) and repeating the proof of Case 1 in Theorem \ref{theo1}, we obtain $x_n\to x^*$.\\
\textbf{Case 2.} There exists a subsequence $\left\{\epsilon_{n_i}\right\}$ of $\left\{x_n\right\}$ such that $\epsilon_{n_i}\le \epsilon_{n_i+1}$ for all $i\ge 0$. \\
It follows from Lemma \ref{M2008} that 
\begin{equation}\label{eq:21a}
\epsilon_{\tau(n)}\le \epsilon_{\tau(n)+1},~\epsilon_n \le \epsilon_{\tau(n)+1},~\forall n\ge n_0.
\end{equation}
where $\tau(n)=\max\left\{k\in N:n_0\le k\le n, ~\epsilon_k\le \epsilon_{k+1}\right\}$. Furthermore, the sequence 
$\left\{\tau(n)\right\}_{n\ge n_0}$ is non-decreasing and $\tau(n)\to +\infty$ as $n\to\infty$.
It follows from (\ref{eq:8a}), the hypothesises 
of $\rho$, $\lim\inf_{n}w_n^i>0$, $\epsilon_{\tau(n)}\le \epsilon_{\tau(n)+1}$ and $\alpha_{\tau(n)}\to 0$ that
\begin{equation}\label{eq:22a}
||{y}^i_{\tau(n)}-x_{\tau(n)}||\to 0,~||{y}^i_{\tau(n)}-{z}^i_{\tau(n)}||\to 0,~||x_{{\tau(n)}+1}-{z}_{\tau(n)}||\to 0,~\forall i\in I.
\end{equation}
Using (\ref{eq:22a}) and repeating the proof of Case 2 in Theorem \ref{theo1}, we obtain $x_n\to x^*$. Theorem \ref{theo2} is proved.
\end{proof}
\section{A numerical example}\label{example}
In this section, we perform a numerical example to illustrate the convergence of Algorithms \ref{Alg1}, \ref{Alg2} and compare them with 
the parallel hybrid extragradient method (PHEM), see \cite[Algorithm 1]{HMA2016}. All programs are written 
in Matlab 7.0 and computed on a PC Desktop Intel(R) Core(TM) i5-3210M CPU @ 2.50GHz 2.50 GHz, RAM 2.00 GB.

We consider the bifunctions $f_i$ which are generalized from the Nash-Cournot equilibrium model in \cite{FP2002,QMH2008} defined by
\begin{equation}\label{fi}
f_i(x,y)=\left\langle P_ix+Q_iy+q_i,y-x\right\rangle,~i\in I=\left\{1,2,\ldots,5\right\},
\end{equation}
where $q_i\in \Re^m$ ($m=10$) and $P_i,~Q_i$ are matrices of order $m$ such that $Q_i$ is symmetric, positive semidefinite and $Q_i-P_i$ is 
negative semidefinite. The feasible set $C\in \Re^m$ is a polyhedral convex set as 
$$ C=\left\{ x\in \Re^m: Ax\le b\right\}, $$
where $A\in \Re^{m\times k}$ is a matrix and $b$ is a positive vector in $\Re^k$ ($k=20$). Let $T_j,~j\in J=\left\{1,2,\ldots,20\right\}$ be half-spaces 
defined by $T_j=\left\{x\in \Re^m:\left\langle x,h_j\right\rangle\le l_j\right\}$, where $h_j\in \Re^m$ and 
$l_j$ are positive real numbers. Define the mappings $S_j:C\to C$ defined by $S_j=P_C P_{T_j}$. The operator $F(x)=x-a$ where 
$a=(1,1,\ldots,1)^T\in \Re^m$. The bifunctions $f_i$ satisfy Condition 1 with $c_1^i=c_2^i=||P_i-Q_i||/2$, see Lemma 6.2 in \cite{QMH2008}. We 
here chose $c_1=c_2=\max\left\{c_1^i:i\in I\right\}$. Since the mappings $S_j$ are nonexpansive, they are $\beta$ - demicontractive with $\beta=0$.
In the mentioned algorithms, we need to solve the following optimization program
$$ \arg\min\left\{\rho f_i(x_n,y)+\frac{1}{2}||x_n-y||^2:y\in C\right\} $$
or the convex quadratic problem
\begin{equation}\label{min}
\arg\min\left\{\frac{1}{2}y^TH_iy+b_i^Ty:y\in C\right\}
\end{equation}
where $H_i=2\rho Q_i+I$ and $b_i=\rho (P_i x_n-Q_i x_n+q_i)-x_n$ to obtain the approximation $y_n^i$. Similarly, 
$z_n^i$ solves the following program
\begin{equation}\label{min1}
\arg\min\left\{\frac{1}{2}y^T\widehat{H}_iy+\widehat{b}_i^Ty:y\in C\right\}
\end{equation}
where $\widehat{H}_i={H}_i$ and $\widehat{b}_i=\rho (P_i y^i_n-Q_i y^i_n+q_i)-x_n$. Problems (\ref{min}), (\ref{min1}) can be 
effectively solved, for instance, by the MATLAB Optimization Toolbox. All projections onto half-spaces are explicit and onto polyhedral convex sets of 
Algorithm 1 in \cite{HMA2016} are rewritten equivalently to convex quadratic problems. 

In below experiments, all entries of $A$, $h_j$ are randomly generated in $[-m,m]$ and of $b$, $l_j$ in $[1,m]$, the vectors $q_i$ are the zero vector.
All entries of $P_i$, $Q_i$ are also generated randomly\footnote{We randomly chose 
$\lambda_{1k}^i\in [-m,0],~\lambda_{2k}^i\in [0,m],~ k=1,\ldots,m,~i=1\ldots,N$. Set $\widehat{Q}_1^i$, $\widehat{Q}_2^i$ as two 
diagonal matrixes with eigenvalues 
$\left\{\lambda_{1k}^i\right\}_{k=1}^m$ and $\left\{\lambda_{2k}^i\right\}_{k=1}^m$, respectively. Then, we make a positive 
semidefinite matrix $Q_i$ and a negative semidefinite matrix $T_i$ by using random orthogonal matrixes with $\widehat{Q}_2^i$ and $\widehat{Q}_1^i$, 
respectively. Finally, set $P_i=Q_i-T_i$} such that they satisfy the mentioned conditions above. It is easy to see that $0\in \cap_{i\in I}EP(f_i,C)$ and $\cap_{j\in J}Fix(S_j)=C\cap (\cap_{j\in J}H_j)$. With 
choosing $b$ and $l_j$ above, then $0\in \cap_{j\in J}Fix(S_j)$, thus $0\in \Omega$. To check whether $\left\{x_n\right\}$ converges to $x^*=0$ or not, 
we use the function $D_n=||x_n-x^*||$ for $n=0,1,2,\ldots$. The convergence of $\left\{D_n\right\}$ to $0$ implies that $\left\{x_n\right\}$ converges to the solution of Problem 
\ref{VIPoverEPandFPP}. We chose the starting point $x_0=(1,1,\ldots,1)^T\in \Re^m$, $\rho=\frac{1}{4c_1}$, $w_n^i=\frac{1}{N}$, 
$\gamma_n^j=\frac{1}{M}$, $\beta_n^j=\frac{1}{4}$ for all $i,~j,~n$. We perform two experiments for all algorithms with $\alpha_n=\frac{1}{(n+1)^{0.5}}$ or 
$\alpha_n=\frac{1}{n+1}$.  Figures \ref{fig1} and \ref{fig2} describe the behavior of $D_n$ with $\alpha_n=\frac{1}{(n+1)^{0.5}}$ and 
$\alpha_n=\frac{1}{n+1}$, resp., for 1000 first iterations. From these figures, we see that the convergence of Algorithm \ref{Alg1} is the best in 
both two cases. In the case $\alpha_n=\frac{1}{n+1}$, the convergence rate of Algorithm \ref{Alg1} is better than the case 
$\alpha_n=\frac{1}{(n+1)^{0.5}}$ and the obtained tolerance is $D_n<10^{-5}$ after 1000 first iterations. The times for execution 
of Algorithm \ref{Alg1} are smaller those of PHEM in two cases. The reason for this is that in Algorithm \ref{Alg1}, we do not need to construct 
two sets $C_n$ and $Q_n$ and find the projection onto their intersection. For Algorithm \ref{Alg2}, although the convergence rate is the slowest, but 
the times for execution is the smallest. This is obvious because in Algorithm \ref{Alg2} we have not to find the furthest approximations and construct two set $C_n$ and $Q_n$ per 
each iteration. This algorithm is the simplest in computing.

\begin{figure}[htb] 
\begin{center}
\includegraphics[width=120mm]{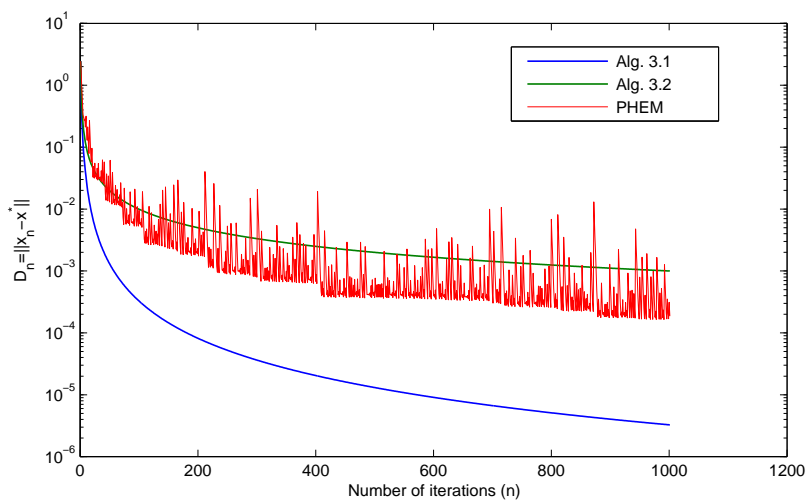}
\caption{Behavior of $D_n=||x_n-x^*||$ for Algorithms \ref{Alg1}, \ref{Alg2} and PHEM with $\alpha_n=\frac{1}{n+1}$ 
(The execution times 1000 first iterations are 58.29s, 49.78s and 80.23s, resp.)} 
\label{fig1}
\end{center}
\end{figure}
\begin{figure}[htb] 
\begin{center}
\includegraphics[width=120mm]{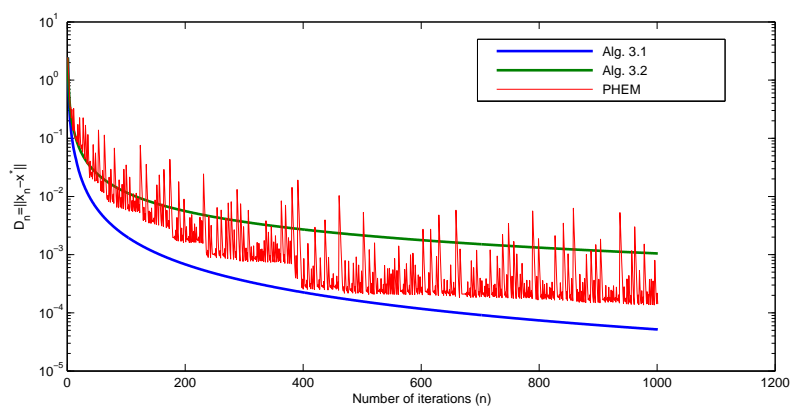}
\caption{Behavior of $D_n=||x_n-x^*||$ for Algorithms \ref{Alg1}, \ref{Alg2} and PHEM with $\alpha_n=\frac{1}{(n+1)^{0.5}}$  
(The execution times for 1000 first iterations are 60.34s, 51.32s and 84.43s, resp.)} 
\label{fig2}
\end{center}
\end{figure}
\section{Concluding} In this paper, we have proposed two parallel extragradient - viscosity methods for finding a particular common solution of 
a finite family of equilibrium problems for pseudomonotone and Lipschitz-type continuous bifunctions and a finite family of fixed point problems 
for demicontractive mappings. The considered particular element is the unique solution of a variational inequality problem on the common solution 
set of two families. The proposed algorithms can be considered as improvements of some previously known hybrid methods in computations. A 
numerical example is performed to illustrate the convergence of the algorithms and compare them with the parallel hybrid extragradient method.


\begin{thebibliography}{99}
\bibitem{A2013} Anh, P. N.: A hybrid extragradient method extended to fixed point problems and equilibrium problems. 
Optimization.  \textbf{62}(2), 271--283 (2013)

\bibitem{ABH2014}Anh, P. K., Buong, Ng., Hieu, D. V.: Parallel methods for regularizing systems of equations involving accretive operators. 
Applicable Analysis: An International Journal. \textbf{93}(10), 2136-2157 (2014)

 \bibitem{AH2014}Anh, P.K., Hieu, D.V: Parallel and sequential hybrid methods for a finite family of asymptotically quasi $\phi$-nonexpansive mappings. 
J. Appl. Math. Comput. \textbf{48}, 241-263 (2015)

 \bibitem{AH2014b} Anh, P.K., Hieu, D.V.: Parallel hybrid methods for variational inequalities, equilibrium problems and common fixed point problems. 
Vietnam J. Math. (2015), DOI:10.1007/s10013-015-0129-z

\bibitem{BO1994} Blum, E., Oettli, W.: From optimization and variational inequalities to equilibrium problems, Math. Program. \textbf{63}, 123-145 (1994)

\bibitem{CH2005} Combettes, P. L., Hirstoaga, S. A.: Equilibrium programming in Hilbert spaces. J. Nonlinear Convex Anal. \textbf{6}, 117-136 (2005)

 \bibitem{CGR2012}Censor, Y., Gibali, A., Reich, S.: Algorithms for the split variational inequality problem. Numer. Algorithms. \textbf{59}(2) (2012), 301-323.

\bibitem{CGRS2012}Censor, Y., Gibali, A., Reich, S., Sabach, S.: Common Solutions to Variational Inequalities. 
Set-Valued Var. Anal. \textbf{20}, 229-247 (2012)

\bibitem{FP2002} Facchinei, F., Pang, J.S.: Finite-Dimensional Variational Inequalities and Complementarity Problems.
Springer, Berlin (2002)

\bibitem{HMA2016} Hieu, D. V., Muu, L. D, Anh, P. K.: Parallel hybrid extragradient methods for pseudomonotone equilibrium problems and 
nonexpansive mappings. Numer. Algorithms. DOI: 10.1007/s11075-015-0092-5

\bibitem{H2015Korean} Hieu, D. V.: A parallel hybrid method for equilibrium problems, variational inequalities and nonexpansive mappings in Hilbert space. 
J. Korean Math. Soc. \textbf{52}, 373-388 (2015)

\bibitem{H2016b}Hieu, D. V.: Parallel hybrid methods for generalized equilibrium problems and asymptotically strictly pseudocontractive mappings. 
J. Appl. Math. Comput. (2016). DOI :10.1007/s12190-015-0980-9.

\bibitem{H2016a} Hieu, D. V.: Parallel extragradient-proximal methods for split equilibrium problems. Math. Model. Anal. (Revised) (2016).  

\bibitem{H2015Iranian} Hieu, D. V.: The common solutions to pseudomonotone equilibrium problems. Bull. Iranian Math. Soc. 
(accepted for publication) (2015) 

\bibitem{H2016} Hieu, D. V.: An extension of hybrid method without extrapolation step to equilibrium problems. Journal of Industrial and 
Management Optimization (Revised). 

\bibitem{K2000} Konnov, I.V.: Combined Relaxation Methods for Variational Inequalities. Springer, Berlin (2000)

\bibitem{KS1980} Kinderlehrer, D., Stampacchia, G.: An Introduction to Variational Inequalities and Their Applications.
Academic Press, New York (1980)

\bibitem{K1976} Korpelevich, G. M.: The extragradient method for finding saddle points and other problems, Ekonomikai Matematicheskie Metody. 
\textbf{12}, 747-756 (1976)

 \bibitem{M2003} Mastroeni, G.: On auxiliary principle for equilibrium problems, in:Equilibrium Problems and Variational Models, P. Daniele et al. (eds), 
 Kluwer Academic Publishers, Dordrecht, 289-298 (2003)

\bibitem{M2008}Maing$\rm\acute{e}$, P. E.: A hybrid extragradient-viscosity method for monotone operators and fixed point
problems, SIAM J. Control Optim. 47 (2008), pp. 1499-1515.

\bibitem{MM2008} Maing$\rm\acute{e}$, P.E., Moudafi A.: Coupling viscosity methods with the extragradient algorithm for solving
equilibrium problems. J. Nonlinear Convex Anal. 2008;9:283-294.

\bibitem{M2009} Moudafi, A. On the convergence of splitting proximal methods for equilibrium problems in
Hilbert spaces. J. Math. Anal. Appl. 2009; 359: 508-513.

\bibitem{QMH2008}Quoc, T.D., Muu,  L.D., Hien, N.V.: Extragradient algorithms extended to equilibrium problems. Optimization \textbf{57}, 749-776 (2008)

\bibitem{R1976}Rockafellar, R.T.: Monotone operators and the proximal point algorithm. SIAM J. Control Optim. \textbf{14}, 877--898 (1976)

\bibitem{VSN2013}Vuong, P.T., Strodiot, J.J., Nguyen, V.H.: On extragradient-viscosity methods for solving equilibrium and fixed point problems in a Hilbert 
space. Optimization (2013) DOI: 10.1080/02331934.2012.759327.

\bibitem{Y2001} Yamada, I.: The hybrid steepest descent method for the variational inequality problem over the intersection of fixed point sets 
of nonexpansive mappings, In: Butnariu, D., Censor, Y., Reich, S. ( eds.)
Inherently Parallel Algorithms for Feasibility and Optimization and Their Applications, Elsevier, Amsterdam, (2001), pp. 473--504.



\end{thebibliography}
\end{document}